\documentclass{amsart}
\usepackage[english]{babel}
\usepackage{amssymb,amsfonts,latexsym,stmaryrd,diagrams,mathrsfs,bbm}
\usepackage{graphicx}

\newcommand{\pair}[2]{\langle #1 , #2 \rangle}
\newcommand{\arr}{\rightarrow}
\newcommand{\Pow}{\mathcal{P}}
\newcommand{\Set}{{\mathbf{Set}}}

\newcommand{\Sub}{{\mathsf{Sub}}}

\def\lbr{\mathopen{{[\kern-0.14em[}}}   
\def\rbr{\mathclose{{]\kern-0.14em]}}}  

\newcommand{\inc}{\hookrightarrow}

\newcommand{\strarr}{\mbox{$\circ \kern-0.4em \arr$}}
\newcommand{\eps}{\,\varepsilon\,}
\newcommand{\neps}{\; \slash \kern-0.7em \eps \,}

\newcommand{\A}{{\mathcal{A}}}

\newcommand{\E}{{\mathcal{E}}}
\newcommand{\F}{{\mathcal{F}}}

\newcommand{\id}{{\mathrm{id}}}

\newcommand{\opp}{{\mathsf{op}}}

\newcommand{\forget}[1]{}

\newcommand{\mono}{\rightarrowtail}
\newcommand{\epi}{\twoheadrightarrow}

\newcommand{\Sh}{\mathsf{Sh}}

\newcommand{\Sm}{{\mathcal{S}}}

\newcommand{\Eff}{\mathcal{E}\kern-0.14em\mathit{ff}}

\newcommand{\Id}{{\mathrm{Id}}}

\newcommand{\Xcal}{{\mathcal{X}}}

\newcommand{\schloop}{{< \kern-0.40em  <}}

\DeclareRobustCommand{\coprod}{\mathop{\text{\fakecoprod}}}
\newcommand{\fakecoprod}{%
  \sbox0{$\prod$}%
  \smash{\raisebox{\dimexpr.9625\depth-\dp0}{\scalebox{1}[-1]{$\prod$}}}%
  \vphantom{$\prod$}%
}

\newarrow{Dashto}    {dash}{dash}{dash}{dash}>
\newarrow{Equal}     =====
\newarrow{Onto}      ----{>>}
\newarrow{Embed}     >--->
\newarrow{Dots}      .....

\newtheorem{thm}{Theorem}[section]

\newtheorem{lem}[thm]{Lemma}

\newtheorem{defn}[thm]{Definition}

\begin{document}

\title{Triposes as a Generalization of Localic Geometric Morphisms}

\author{Jonas Frey and Thomas Streicher}


\title{Triposes as a Generalization of Localic Geometric Morphisms}

\begin{abstract}
In \cite{HJP} Hyland, Johnstone and Pitts introduced the notion of 
\emph{tripos} for the purpose of organizing the construction of realizability 
toposes in a way that generalizes the construction of localic toposes from 
complete Heyting algebras. In \cite{Pit02} one finds a generalization of this 
notion eliminating an unnecessary assumption of \cite{HJP}.

The aim of this paper is to characterize triposes over a base topos $\Sm$
in terms of so-called \emph{constant objects} functors from $\Sm$ to some
elementary topos. Our characterization is slightly different from the one in 
Pitts's PhD Thesis \cite{Pit81} and motivated by the fibered view of geometric 
morphisms as described in \cite{Str}. In particular, we discuss the question
whether triposes over $\Set$ giving rise to equivalent toposes are already 
equivalent as triposes.
\end{abstract}

\maketitle

\begin{center}
\emph{Dedicated to the Memory of Martin Hofmann}
\end{center}

\section{Introduction}

As described in \cite{Joh} localic geometric morphisms to a topos $\Sm$
are given by functors $F$ from $\Sm$ to some topos $\E$ such that
\begin{itemize}
\item[(1)] $F$ preserves finite limits,
\item[(2)] every object $A\in\E$ appears as subquotient of some $FI$ and
\item[(3)] $F$ has a right adjoint.
\end{itemize}
In the appendix of \cite{Str} one finds a proof of M.~Jibladze's Theorem \cite{Jib}
saying that fibered toposes over $\Sm$ having internal sums correspond to 
finite limit preserving functors from $\Sm$ to some topos $\E$. In particular, a 
finite limit preserving functor $F : \Sm\to\E$ gives rise to the fibration 
$P_F = F^*P_\E$ over $\Sm$ obtained by change of base along $F$ from the
fundamental (``codomain'') fibration $P_\E = \mathsf{cod} : \E^{\mathbbm{2}}\to \E$
(where $\mathbbm{2}$ is the small category $0 \to 1$ corresponding to the ordinal $2$).
But every fibered topos $P : \Xcal\to\Sm$ with internal sums is equivalent to $P_\Delta$
where $\E$ is the fiber of $P$ over $1$ and $\Delta : \Sm\to\E$ sends 
$I\in\Sm$ to $\Delta(I) = \coprod_I 1_I$ in $\E$.

Moreover, as also shown in \cite{Str} for terminal object preserving $F :
\Sm\to\E$ the fibration $P_F$ is locally small iff $F$ has a right adjoint.
Thus, as  observed by J.~B\'enabou already in his 1974 Montreal
lectures~\cite{benabou1974logique}, inverse image parts of geometric morphisms
correspond to terminal object preserving functors $F$ between toposes such that
the fibration $P_F$ has internal sums and is locally small.

Moreover, as also observed in \cite{Str} for a finite limit preserving functor
$F : \Sm\to\E$ between toposes condition (2) is equivalent to the requirement that
every map $a : A \to FI$ in $\E$ fits into a commuting diagram
\begin{diagram}
  C & \rOnto^e & A \\
  \dEmbed^{m} & & \dTo_a \\
  FJ & \rTo_{Fu} & FI
\end{diagram}
where $e$ is epic and $m$ is monic. Obviously, this condition entails (2) instantiating
$I$ by a terminal object. For the reverse direction choose $m : C \mono FJ$ and
$e : C \epi A$ (which exist by condition (2)) and observe that
\begin{diagram}
  & & C & \rOnto^e & A \\
  & \ldEmbed^{m} & \dEmbed_{\pair{m}{ae}} & & \dTo_a \\
  FJ & \lTo_{F\pi_1} & F(J{\times}I) & \rTo_{F\pi_2} & FI
\end{diagram}
using the assumption that $F$ preserves finite limits and thus finite products.
Thus, condition (2) amounts to the requirement that every object of $\E$ can be covered
by a(n internal) sum of subterminals (in the appropriate fibrational sense!). 
As observed in \cite{Str} under assumption (3) this is equivalent to the requirement
that $g$ in
\begin{diagram}
  G \SEpbk & \rTo & 1_\E \\
  \dEmbed^g & & \dEmbed_{\top_\E} \\
  FU\Omega_\E & \rTo_{\varepsilon_{\Omega_\E}} & \Omega_\E
\end{diagram}
is a generating family for the fibration $P_F$ (where $U$ is right adjoint to $F$).

\section{A fibrational account of triposes}

In \cite{HJP} Hyland, Johnstone and Pitts have identified a notion of fibered preorder 
$\mathscr{P}$ over a base topos
$\Sm$ giving rise to a topos $\Sm[\mathscr{P}]$ by ``adding subquotients''
related to the base topos via a \emph{constant object functor}
$\Delta_{\mathscr{P}} : \Sm \to \Sm[\mathscr{P}]$ satisfying conditions (1) and
(2) of the previous section. As obvious from the considerations in
\emph{loc.cit.}\ one may get back the fibered preorder $\mathscr{P}$ from the
subobject fibration $\Sub_{\Sm[\mathscr{P}]}$ by change of base along
$\Delta_{\mathscr{P}}$. For this reason such $\mathscr{P}$ were called
``\underline{t}opos \underline{r}epresenting \underline{i}ndexed
\underline{p}re-\underline{o}rdered \underline{s}ets'' suggesting the acronym
``tripos'' (echoing the traditional name for final exams at University of
Cambridge). 

The original definition in \cite{HJP} required triposes to have a \emph{generic family}, i.e.\ a
$T \in \mathscr{P}(\Sigma)$ from which all objects in $\mathscr{P}(I)$ may be
obtained (up to isomorphism) by reindexing along an appropriate (generally not unique)
map $I\to\Sigma$.
In \cite{Pit02}, Pitts observed that the requirement of a generic family can be replaced by a weaker 
`comprehension axiom'~\cite[Axiom~4.1]{Pit02} which still implies -- and is actually equivalent to -- the fact that the category
$\Sm[\mathscr{P}]$ of partial equivalence relations is a topos\footnote{The definition of tripos in~\cite{Pit02} 
generalizes that of~\cite{HJP}
also 
in the sense that it admits arbitrary finite-product categories as 
base categories, and requires the Beck-Chevalley condition (BC) only for
certain pullback squares definable from finite products. 
However, in the present work we consider only toposes as base categories, and all triposes
satisfy BC for arbitrary pullback squares.}.

In this work we focus on triposes in this more general sense, and contrary to the literature we use the word `tripos' not for the fibered preorder, but for the associated constant objects functor (from which the fibered preorder can be reconstructed as pointed out above). This way we obtain a definition that is a straightforward generalization of the notion of localic geometric morphism as presented in the introduction. Pitts's comprehension axiom is then presented as Lemma~\ref{hopitts}. We refer to (constant objects functors arising from) the more restrictive notion of~\cite{HJP} as ``traditional triposes'':

\begin{defn}\label{tripdef}
A \emph{tripos} over a topos $\Sm$ is a finite limit preserving functor $F$ from $\Sm$
to a topos $\E$ such that every $A\in\E$ appears as subquotient of $FI$ for some $I\in\Sm$.

A tripos is called \emph{traditional} if the fibered preorder $\mathscr{P}_F=F^*\Sub_\E$ admits a generic family, i.e.\ there is a 
a mono $t : T \mono F\Sigma$ such that every mono $m : P \mono FI$ fits into a
pullback diagram
\begin{diagram}
  P \SEpbk & \rTo & T \\
  \dEmbed^m & & \dEmbed_t \\
  FI & \rTo_{Fp} & F\Sigma
\end{diagram}
for some (typically not unique) $p : I\to\Sigma$. 
\end{defn}
The fibered preorder $\mathscr{P}_F$ is a first-order hyperdoctrine, i.e.\ a
fibration of Heyting algebras with internal sums and products, since $\Sub_\E$
has and change of base along the finite limit preserving functor $F$ preserves
the required properties (see \cite{Str}). Recalling Pitts's proof of the
comprehension axiom, we now show that $\mathscr{P}_F$ admits an interpretation
higher order (intuitionistic) logic.

\begin{lem}\label{hopitts} For every tripos $F : \Sm\to\E$, the fibered preorder
$\mathscr{P}_F$ satisfies the following \emph{comprehension axiom}.
\begin{itemize}
\item[(CA)]\label{CA} For every object $I$ in $\Sm$ there is are objects
$P(I)$ in $\Sm$ and $\varepsilon_I \in \mathscr{P}_F(I{\times P(I)})$ such that
for all $J$ in $\Sm$ and $\rho \in \mathscr{P}_F(I{\times}J)$,
the formula
\[
\forall j \in J. \exists p \in P(I).\forall i \in I.\, 
   \rho(i,j) \leftrightarrow i \,\varepsilon_I\, p 
\]
holds in the internal logic of $\mathscr{P}_F$.
\end{itemize}
\end{lem}
\begin{proof}
Let $I$ be an object of $\Sm$. Then since $\mathscr{P}$ is a tripos there is an object
$P(I)$ in $\Sm$ such that $\Pow(FI)$ appears as subquotient of $F(P(I))$, i.e.\ there
is a subobject $m_I : C_I \mono F(P(I))$ such that there exists an epi
$e_I : C_I \epi \Pow(FI)$. Consider

\begin{diagram}
  \varepsilon_I & \lEqual & \varepsilon_I \SEpbk & \rOnto & \in_{FI} \\
  \dEmbed & & \dEmbed & & \dEmbed \\
  FI \times  F(P(I)) & \lEmbed_{FI \times m_I} & FI \times C_I & \rOnto_{FI \times e_I} & FI \times \Pow(FI)
\end{diagram}
giving rise to a subobject $\varepsilon_I$ of $F(I{\times}P(I)) \cong  FI{\times }F(P(I))$.
Since the left square in the above diagram is a pullback we have
$(FI \times m_I)^*{\varepsilon_I} = (FI \times e_I)^*{\in_{FI}}$.

Suppose $\rho : R \mono F(I{\times}J) \cong FI \times FJ$. Then
\begin{diagram}
  R \SEpbk & \rTo & \in_{FI} \\
  \dEmbed^\rho & & \dEmbed \\
  FI{\times}FJ  & \rTo_{FI \times r} & FI{\times}\Pow(FI)
\end{diagram}
for  a unique $r : FJ \to \Pow(FI)$. Consider the pullback
\begin{diagram}
  C \SEpbk & \rTo^{\widetilde{r}} & C_I \\
  \dOnto^e & & \dOnto_{e_I} \\
  FJ & \rTo_r & \Pow(FI)
\end{diagram}
where $e$ is epic since in a topos epis are stable under arbitrary pullbacks.
Thus, we have
\begin{tabbing}
  \qquad \= $(FI \times e)^*\rho$ \= $= (FI \times e)^*(FI \times r)^*{\in_{FI}}
    \;\cong (FI \times er)^*{\in_{FI}} =$ \\
    \>\> $= (FI \times e_I\widetilde{r})^*{\in_{FI}}
    \cong (FI \times \widetilde{r})^*(FI \times e_I)^*{\in_{FI}} =$ \\
    \>\> $= (FI \times \widetilde{r})^*(FI \times m_I)^*{\varepsilon_I} \cong$ \\
    \> $\cong (FI \times m_I\widetilde{r})^*\varepsilon_I$
\end{tabbing}
from which it readily follows that
\[ \forall j \in J. \exists p \in P(I).\forall i \in I.\, 
   \rho(i,j) \leftrightarrow i \,\varepsilon_I\, p \]
holds in the internal logic of $\mathscr{P}_F$.
\end{proof}

From Lemma~\ref{hopitts} and the results of \cite{Pit02} it follows that fibered preorders
of the form $F^*\Sub_\E$ for some tripos $F : \Sm\to\E$ may up to equivalence be
characterized as Heyting algebras $\mathscr{P}$ fibered over $\Sm$ with internal sums
$\exists$ and internal products $\forall$ satisfying the comprehension axiom~(CA).


A tripos $F : \Sm\to\E$ is traditional iff 
$\mathscr{P}_F$ is a tripos in the sense of \cite{HJP}, i.e.\ there exists a $
T \in \mathscr{P}(\Sigma)$ from which all $P \in \mathscr{P}(I)$ can be obtained
by reindexing along some map $p : I\to\Sigma$. 

If a tripos $F$ has a right adjoint (and thus is the inverse image part of a localic geometric morphism) then it is always traditional since we can set $\Sigma=U\Omega_\E$, and in this case the maps $p$ are unique.

Finally, we note that $F$ 
is the inverse image part of a localic geometric morphism iff $F^*\Sub_\E$ is locally small iff $P_F$ is locally small (see \cite{Str}).

\section{Constant objects functors are not unique}

For arbitrary base toposes $\Sm$ triposes $F,G : \Sm\to\E$ need not be equivalent
since if $\Sm$ is $\Sh(X)$ and $\E$ is $\Sh(Y)$ for some sober spaces $X$ and $Y$
then there are at least as many triposes $\Sm\to\E$ (up to equivalence) as there
are continuous maps from $Y$ to $X$. But even if $\Sm$ is $\Set$ there are in general
many non-equivalent triposes over $\Set$ giving rise to the same topos as shown by the
following simple counterexample

\begin{thm}
For every natural number $n>0$ the functor $F_n : \Set\to\Set : I \mapsto I^n$ is a tripos.
The triposes $F_n^*\Sub_\Set$ and $F_m^*\Sub_\Set$ are equivalent if and only if $n=m$.
\end{thm}
\begin{proof}
Obviously, the $F_n$ preserve finite limits since they are right adjoints and every
$I\in\Set$ appears as split subobject of $F_n(I)$. Thus, all $F_n$ are triposes but
$F_n^*\Sub_\Set$ and $F_m^*\Sub_\Set$ are equivalent as triposes if and only if $n=m$
since the latter is equivalent to $2^n = 2^m$ which in turn is equivalent to
$F_n^*\Sub_\Set(2) \simeq F_m^*\Sub_\Set(2)$.
\end{proof}

Notice, however, that $F_n$ is a traditional tripos if and only if $n=1$. Thus, it may
still be the case that there exist traditional triposes $F,G : \Sm \to \E$ which are not
equivalent as triposes. Unfortunately, we have not been able so far to find examples of
non-equivalent \emph{traditional} triposes $\mathscr{P}_1$ and $\mathscr{P}_2$ over $\Set$
such that the ensuing toposes $\Set[\mathscr{P}_1]$ and  $\Set[\mathscr{P}_2]$ are
equivalent (a question which for the special case of localic toposes has been raised already in~\cite[p.~228]{HJP}). However, though a bit annoying, we can't find this as a major problem
since our weak notion of tripos is conceptually more adequate than the traditional one
because from a logical point of view adding the comprehension axiom to
first-order posetal hyperdoctrines appears much more natural than requiring that
they are witnessed by Skolem functions in the base $\Sm$, i.e.\ requiring for
all $\rho \in \mathscr{P}(I{\times}J)$ the existence of a function $r : J \to
P(I)$ such that
\[ \forall j \in J.\forall i \in I.\, \rho(i,j) \leftrightarrow i \,\varepsilon_I\, r(j) \]
holds in the logic of $\mathscr{P}$. At the end of \cite{Pit02} the author expresses
a similar view in a slightly more cautious way.

Finally, we observe that triposes over $\Set$ may give rise to non-localic
Gro\-then\-dieck toposes. Let $\E$ be the topos of reflexive graphs, i.e.\
presheaves over the 3 element monoid $\Delta([1],[1])$ of monotone endomaps of
the ordinal $2$. As observed by Lawvere the global elements functor $\Gamma : \E
\to\Set$ fits into a sequence of adjoints $\Pi \dashv \Delta \dashv \Gamma
\dashv \nabla : \Set \inc \E$. The rightmost functor $\nabla$ preserves all
limits since it has a left adjoint. Subobjects of objects of the form
$\nabla(I)$ are up to isomorphism precisely those reflexive graphs where between
two nodes there is at most one edge (i.e.\ directed graphs as traditionally
considered in combinatorics!). But since any reflexive graph can be covered by
such a traditional directed graph every object of $\E$ appears as subquotient of
some $\nabla(I)$ for which reason $\nabla$ is a tripos over $\Set$ though it is
not the inverse image part of a geometric morphism.

\section{Regular triposes}

It is well known that a morphism $e : Y\to X$ in an elementary topos $\E$ is epic iff
the pullback functor $e^* : \Sub_\E(X) \to \Sub_\E(Y)$ reflects maximal subobjects, i.e.\
a mono $m : P \mono X$ in $\E$ is an iso already if $e^*m$ is an iso. Recall that a preorder
fibered over a regular category is a \emph{prestack} (w.r.t.\ the regular cover topology)
iff for all regular epis $e$ reindexing along it (preserves and) reflects the order.
Thus, for a tripos $F : \Sm\to\E$ the fibered preorder $F^*\Sub_\E$ is a prestack iff $F$
preserves (regular) epis. 

This observation strongly suggests to require that triposes $F : \Sm\to\E$ also preserve 
epis since it vacuously holds when $\Sm$ is $\Set$ (since in $\Set$ all epis are split 
as ensured by the axiom of choice!) and, moreover, by
Lemma~6.1 (``Pitts's Iteration Theorem'') of \cite{Pit81} 
triposes preserving epis are closed under composition.

\begin{defn}
A tripos $F : \Sm\to\E$ is called \emph{regular} iff $F$ preserves epis.
\end{defn}

Recall that a functor between regular categories is called \emph{regular} iff it
preserves finite limits and regular epis. Thus, regular triposes are regular functors
$F : \Sm \to\E$ between toposes such that every $A\in\E$ appears as subquotient of
$FI$ for some $I\in\Sm$. The usual proof (as in \cite{Joh}) that localic geometric 
morphisms are closed under composition extends straightforwardly to an argument
showing that regular triposes are closed under composition 

From Prop.~3.14 of \cite{Pit81} it follows that a traditional tripos  
$F : \Sm\to\E$ is regular iff it has ``fibrewise quantification'', i.e.\ 
there are maps $\bigvee,\bigwedge : \Omega_\Sm^\Sigma \to\Sigma$ such that
$\exists_{Fu}(Fp)^*t$ and $\forall_{Fu}(Fp)^*t$ appear as pullbacks
of $t : T\mono F\Sigma$ along $F(\lambda i{:}I.\bigvee \{ p(j) \mid u(j) = i\})$ 
and $F(\lambda i{:}I.\bigwedge \{ p(j) \mid u(j) = i\})$, respectively,
for all $u:J \to I$ and $p : J\to\Sigma$.

\begin{thm}\label{iteration1}
Let $F_1 : \Sm\to\E_1$ and $F_2 : \Sm\to\E_2$ be triposes and $H : F_1 \to F_2$,
i.e.\ $H : \E_1\to\E_2$ with $F_2 = HF_1$. Then $H$ is a tripos iff $H$ preserves finite
limits and $H$ is a regular tripos iff $H$ is a regular functor.
\end{thm}
\begin{proof}
The forward directions are trivial. For the backwards directions suppose $A\in\E_2$.
Then, since $F_2$ is a tripos there exists a subobject
$m : C \mono F_2I$ and an epi $e : C \epi A$. Since $F_2 = H F_1$ we have
$m : C \mono H(F_1I)$ and $e : C \epi A$. Thus, we have shown that $H$ validates
the second condition required for a tripos.
\end{proof}

The previous theorem for regular triposes $F_1 : \Sm\to\E_1$ and $F_2 : \Sm\to\E_2$
suggests that the right notion of morphism from $F_1$ to $F_2$ is a functor
$H : F_1 \to F_2$ such that $H : \E_1 \to \E_2$ is regular since for this definition
morphisms to a regular tripos $F : \Sm\to\E$ coincide with regular triposes over $\E$.\footnote{Moreover, for a not necessarily regular tripos $F : \Sm \to \E$ a regular functor
$H$ from $\E$ to a topos $\F$ is a regular tripos whenever $H \circ F : \Sm\to\F$
is a regular tripos.}

In the subsequent Theorem~\ref{iteration2} we will show that morphisms between
traditional regular triposes are precisely the traditional regular triposes.
But for this purpose we need the following lemma characterizing traditional regular
triposes among regular triposes in terms of a condition which at first sight looks
weaker than the one given in Def.~\ref{tripdef}.

\begin{lem}\label{FreyLem}
Let $F : \Sm\to\E$ be a regular tripos and $t : T \mono F\Sigma$ be weakly generic
for $F^*\Sub_\E$, i.e.\ every mono $m : P \mono FI$ fits into a diagram
\begin{diagram}
P & \lTo & \bullet & \rTo & T \\
\dEmbed^m & & \dEmbed & & \dEmbed_t \\
FI & \lTo_{Fe} & FJ & \rTo_{Fp} & F\Sigma \\
\end{diagram}
where both squares are pullbacks and $e : J \to I$ is epic, then $F$
is a traditional tripos.
\end{lem}
\begin{proof}
Suppose $t : T \mono F\Sigma$ is a weakly generic family for $F^*\Sub_\E$.
Let $E = \{ (u,U) \in \Sigma \times P(\Sigma) \mid u \in U \}$ and
$p : E \to \Sigma$ and $q : E \to P(\Sigma)$ the respective projection maps. 
We will show that $\exists_q p^* t$ is a generic family for $F^*\Sub_\E$.

For this purpose suppose $m \in \Sub_\E(FI)$. By assumption there are
$e : J \epi I$ and $f : J \to \Sigma$ such that $e^*m \cong f^*t$. Since
$F^*\Sub_\E$ is a prestack w.r.t.\ the regular cover topology we have
$m \cong \exists_ee^*m \cong \exists_ef^*t$. Let $g : I \to P(\Sigma)$
with $g(i) = \{ f(j) \mid e(j) = i\}$. Obviously, the map $\pair{f}{ge}$
factors through $\pair{p}{q}$ since $f(j) \in g(e(j))$. 
Consider the following diagram
\begin{diagram}
& & J & & \\
&  \ldTo(2,4)^f & \dOnto_h & \rdOnto^e & \\
& & R \SEpbk & \rTo_{r_2} & I \\
& & \dTo_{r_1} & & \dTo_g \\
\Sigma & \lTo_p & E & \rTo_q & P(\Sigma)
\end{diagram}
where 
$R = \{ (u,U,i) \in \Sigma \times P(\Sigma) \times I \mid u \in U = g(i) \}$
with $r_1$ and $r_2$ the respective projections and $h(j) = (f(j),g(e(j)),e(j))$.
Notice that $h$ is onto since $g(i) = \{ f(j) \mid e(j) = i \}$. We have
\begin{tabbing}
  \qquad \= $g^*\exists_qp^*t$ \= $\cong$ \= $\exists_{r_2}r_1^*p^*t$
  \qquad\qquad \= by Beck-Chevalley condition\\
  \> \> $\cong$ \> $\exists_{r_2}\exists_hh^*r_1^*p^*t$ \> since $h$ is epic \\
  \> \> $\cong$ \> $\exists_ef^*t$ \> since $e = r_2h$ and $f = pr_1h$ \\
  \>\> $\cong$ \> $m$
\end{tabbing}
as desired.
\end{proof}

\begin{thm}\label{iteration2}
Let $F_1 : \Sm \to \E_1$ be a traditional regular tripos and $H : \E_1 \to \E_2$
a regular functor between toposes. Then $H$ is a traditional regular tripos if and
only if $F_2 = H F_1$ is a traditional regular tripos.
\end{thm}
\begin{proof}
The forward direction is Pitts's Iteration Theorem.

For the backward direction suppose that $H : \E_1 \to \E_2$ is a regular functor
such that $F_2 = H F_1$ is a traditional regular tripos. By Theorem~\ref{iteration1}
it is immediate that $H$ is a regular tripos, too. Since $F_2$ has been assumed
to be a traditional tripos there is a $t : T \mono F_2\Sigma$ generic for $F_2^*\Sub_{\E_2}$.
For showing that $H$ is a traditional tripos it suffices by Lemma~\ref{FreyLem} to show
that $t : T \mono H F_1 \Sigma$ is weakly generic for $H^*\Sub_{\E_2}$.

Suppose $m : P \mono HA$ for some $A \in \E_1$. Since $F_1$ is a traditional tripos
there exist $n : Q \mono F_1I$ and $e : Q \epi A$ for some $I  \in \Sm$. Since
$F_2 = H F_1$ is a traditional tripos there exists $p : I \to \Sigma$ such that
$Hn \circ (He)^*m$ arises as pullback of $t$ along $H F_1 p$ for some $p : I \to \Sigma$.
Thus we have
\begin{diagram}
  P & \lOnto & \bullet & \rTo &  T \\
  \dEmbed^m & & \dEmbed_{(He)^*m} & & \\
  HA & \lOnto_{He} & HQ & & \dEmbed_{t} \\
  & & \dEmbed_{Hn} & & \\
  & & HF_1I & \rTo_{HF_1 p} & H F_1 \Sigma
\end{diagram}
from which it follows that $(He)^*m$ arises as pullback of $t$ along
$H(F_1p \circ n) = HF_1p \circ Hn$. Thus, we have
\begin{diagram}
  P & \lOnto & \bullet & \rTo &  T \\
  \dEmbed^m & & \dEmbed_{(He)^*m} & & \dEmbed_{t} \\
  HA & \lOnto_{He} & HQ & \rTo_{H(F_1 p \circ n)} &  H F_1 \Sigma \\
\end{diagram}
where both squares are pullbacks as required. 
\end{proof}

We conclude this section with some observations on the

\subsection{Preservation of assemblies by tripos morphisms}\label{assembpres}

Following \cite{realizbook} for a tripos $F : \Sm\to\E$ one may define
\emph{assemblies} as those objects of $\E$ which appear as subobjects of some
$FI$. If $G : \Sm\to\F$ is a tripos and $H : F \to G$ such that {$H :
\E\to\F$} preserves finite limits then $H$ preserves assemblies, i.e.\ sends
assemblies w.r.t.\ $F$ to assemblies w.r.t.\ $G$, since $Hm : HP \mono HFI = GI$
whenever {$m : P \mono FI$}. It follows from the definition of tripos that
every object $A$ of $\E$ appears as subquotient of some $FI$, i.e.\ we have
$\begin{diagram} A & \lOnto^e & C & \rEmbed^m & FI \end{diagram}$. If {$H :
F \to G$} is a regular functor between triposes then $\begin{diagram} HA &
\lOnto^{He} & HC & \rEmbed^{Hm} & HFI = GI \end{diagram}$, i.e.\ $H$ preserves
coverings of objects by assemblies in a very strong sense.

\section{Relation to Miquel's implicative algebras}

In \cite{Miq18} A.~Miquel has shown that traditional triposes over $\Set$
correspond to so called \emph{implicative algebras}~\cite{Miq20}.

\begin{defn}
An \emph{implicative structure} is a complete lattice $\A$ together with an
operation $\to\; : \A^\opp \times \A \to \A$ such that $x \to \bigwedge Y =
\bigwedge\limits_{y \in Y} (x \to y)$ for all $x \in \A$ and $Y \subseteq \A$.
Then $K_\A = \bigwedge\limits_{x,y \in\A} x{\to}y{\to}x$ and $S_\A =
\bigwedge\limits_{x,y,z \in\A} (x{\to}y{\to}z){\to}(x{\to}y){\to}x{\to}z$ are
elements of $\A$.

A \emph{separator} in an implicative structure $(\A,\to)$ is an upward closed
subset $\Sm$ of $\A$ such that $K_\A,S_\A \in \Sm$ and $\Sm$ is closed under
\emph{modus ponens}, i.e.\ $b \in \Sm$ 
whenever $a \in \Sm$ and $a \to b \in \Sm$.

An \emph{implicative algebra} is a triple $(\A,\to,\Sm)$ such that 
$(\A,\to)$ is an implicative structure and $\Sm$ is a separator in $(\A,\to)$.
\end{defn}

With every implicative algebra $\A$ one associates a $\Set$-based tripos 
$\mathscr{P}^\A$ where $\mathscr{P}^\A(I)$ is the preorder $\vdash_I$ on $\A^I$
defined as
\[ \varphi \vdash_I \psi  \quad \mbox{iff} \quad
   \bigwedge\limits_{i\in I} \bigl(\varphi_i \to \psi_i\bigr) \in \Sm \]
and reindexing is given by precomposition.

In \cite{Miq18} A.~Miquel has shown
that every traditional regular tripos over $\Set$ is equivalent to
$\mathscr{P}^\A$ for some implicative algebra $\A$.

For $i{=}1,2$ let $F_i : \Set \to \E_i$ be the constant objects functor for the
regular tripos induced by an implicative algebra $\A_i$ in $\Set$, i.e.\ $\E_i =
\Set[\mathscr{P}^{\A_i}]$. Due to the remark in Subsection~\ref{assembpres}
regular functors $G : \E_1 \to \E_2$ with $F_2 = G F_1$ correspond to cartesian
functors $g : F_1^*\Sub_{\E_1} \to F_2^*\Sub_{\E_2}$ preserving regular logic,
i.e.\ finite limits and existential quantification. Obviously, such $g$ are
uniquely determined by $h = g_{\A_1}(\id_{\A_1}): \A_1 \to \A_2$ since
$g_I(\varphi : I \to \A_1) = h \circ \varphi$. This suggests to define a
morphism of implicative algebras from $\A_1$ to $\A_2$ as a function $h : \A_1
\to \A_2$ such that the cartesian functor $g : F_1^*\Sub_{\E_1} \to
F_2^*\Sub_{\E_2}$ given by $g_I(\varphi : I \to \A_1) = h \circ \varphi$
preserves regular logic, i.e.\ finite limits and existential quantification.

Unfortunately, Miquel's result from \cite{Miq18} does not extend to 
arbitrary base toposes. The reason is that for a traditional regular 
tripos $F : \Sm \to \E$ there need not exist a subobject $S$ of $\Sigma$
such that
\begin{enumerate}
\item[(1)] its characteristic map $\chi_S : \Sigma\to\Omega_\Sm$ induces
  by postcomposition a cartesian functor $\gamma_S : F^*\Sub_\E \to \Sub_\Sm$
  preserving finite meets in each fiber and
\item[(2)] $u : 1 \to \Sigma$ factors through $S$ iff $(Fu)^*t$ is 
  isomorphic to $\id_{F1}$.
\end{enumerate}
Notice that the first condition means that $u^*S \leq v^*S$ whenever $F(u)^*t
\leq F(v)^*t$ and that $t \in S$ and $\forall u,v:\Sigma.\,u \wedge v \in S
\leftrightarrow (u \in S \wedge v \in S)$ hold in the internal logic of $\Sm$.
For base toposes $\Sm$ which are not well-pointed such $S$ need neither exist
nor be unique (for the latter see Example~4.12.12 of \cite{FreyThes} for a
counterexample\footnote{Take for $\Sm$ the Sierpi\'nski topos
$\Set^{\mathbbm{2}^\opp}$ and for $F$ the functor $\Id_\Sm$. Then there are two
possible choices for $S$, namely $\top : 1 \to \Omega_\Sm$ and the subobject $S$
of $\Omega_\Sm$ with $S_0 = \Omega_0$ and $S_1 = \{\top\}$. In the first case
the corresponding $\gamma_S$ is $\id_{\Omega_\Sm}$ and in the second case it
sends a subobject $P$ of $A$ in $\Sm$ to the subobject $\gamma_S(P)$ of $A$ with
$\gamma_S(P)_0 = A_0$ and $\gamma_S(P)_1 = P_1$.}).

The related stronger condition that $u : I\to\Sigma$ factors through $S$ iff
$F(u)^*t$ is isomorphic to $\id_{FI}$ is known as ``definability of truth'',
i.e.\ that the full subfibration of $F^*\Sub_\E$ on true predicates is definable
in the sense of Bénabou (see Section 12 of \cite{Str}). This stronger condition,
however, amounts to the requirement that the fibration $F^*\Sub_\E$ is locally
small, i.e.\ equivalent to the externalization of a complete Heyting algebra
internal to $\Sm$, which in turn is equivalent to the requirement that $F$ is
the inverse image part of a localic geometric morphism.

\section{Summary and Conclusion}

We have shown in which sense (generalized) triposes in the sense of \cite{Pit81}
may be understood as a generalizations of localic geometric morphisms. The
traditional triposes of \cite{HJP} can be characterized as those triposes $F :
\Sm\to\E$ for which the fibered preorder $F^*\Sub_\E$ admits a generic family $t
: T \mono F\Sigma$.

We have defined regular triposes as triposes $F : \Sm\to\E$ where $F$ preserves
epis, i.e.\ $F^*\Sub_\E$ is a prestack. As opposed to ordinary triposes regular
triposes are known to be closed under composition, i.e.\ are closed under
iteration. A further advantage of regular triposes is that for a regular tripos
$F : \Sm\to\E$ regular triposes over $\E$ correspond to morphisms of regular
triposes from $F$ to some regular tripos $G : \Sm\to\F$, i.e.\ $H : F \to G$
such that $H : \E\to\F$ is a regular functor. Somewhat surprisingly, an
analogous result holds for traditional regular triposes as well. 

Finally, we have recalled a theorem due to A.~Miquel characterizing traditional
regular triposes in terms of implicative algebras generalizing the notion of
complete Heyting algebra and identified a notion of morphism between implicative
algebras corresponding to regular morphisms of triposes over $\Set$.

We think that the more general notion of tripos as introduced in \cite{Pit02} is
more natural since it corresponds to the class of first-order posetal
hyperdoctrines which give rise to toposes by ``adding subquotients''. Moreover,
the comprehension axiom characterizing them is more natural than the Skolemized
form postulated as an axiom in the definition of traditional triposes.

But restricting to regular triposes seems to be a good idea since the condition
is most natural from the point of view of fibered categories and, moreover,
allows one to identify regular tripos morphisms to $F : \Sm\to\E$ with regular
triposes over $\E$ as shown in Theorem~\ref{iteration1}.

We have shown that triposes $F,G : \Set\to\E$ need not be equivalent. But we do
not know whether such $F$ and $G$ are necessarily equivalent under the stronger
assumption that both $F$ and $G$ are traditional triposes. There is no
conceptual reason why this should hold in general but, alas, we have not been
able to find a counterexample so far.

\subsection*{Acknowledgements}

We thank A.~Miquel for making an early version of \cite{Miq18} available to us.
The second named author thanks S.~Maschio for discussions which have triggered
the identification of the right notion of morphism between triposes. We further
acknowledge the use of Paul Taylor's diagram macros used for writing this paper.
The first named author gratefully acknowledges support by the Air Force Office
of Scientific Research through grant FA9550-20-1-0305 and MURI grant
FA9550-15-1-0053.

\newpage
\nocite{benabou1980}
\bibliographystyle{amsalpha}
\bibliography{biblio}

\end{document}